\documentclass[reqno,11pt]{amsart}
\usepackage{amssymb,amsmath,amsfonts,mathrsfs}

\numberwithin{equation}{section} \theoremstyle{plain}
\theoremstyle{plain}
\newtheorem{Thm}{Theorem}
\newtheorem{Cor}[subsection]{Corollary}
\newtheorem{Lem}[subsection]{Lemma}
\newtheorem{Prop}[subsection]{Proposition}

\theoremstyle{definition}
\newtheorem{Def}[subsection]{Definition}

\theoremstyle{remark}

\newtheorem{rem}[subsection]{Remark}

\theoremstyle{example}
\newtheorem{ex}[subsection]{Example}

\newenvironment{thm}%
          { \begin{Thm}  }%
          { \end{Thm} }

\newenvironment{lem}%
          { \begin{Lem}    }%
          { \end{Lem} }

          { \begin{Prop}  }%
          { \end{Prop} }

\newenvironment{cor}%
          { \begin{Cor} }%
          { \end{Cor} }

          { \begin{Def} }%
          { \end{Def} }

\newcommand{\loc}{\operatorname{loc}}
\newcommand{\comp}{\operatorname{comp}}
\newcommand{\mcomp}{C_{c}^{\infty}(M)}
\newcommand{\ecomp}{C_{c}^{\infty}(E)}
\newcommand{\fcomp}{C_{c}^{\infty}(F)}
\newcommand{\End}{\operatorname{End}}
\newcommand{\symd}{\widehat{D}}
\newcommand{\symdt}{\widehat{D^*}}
\newcommand{\hmax}{H_{\max}}
\newcommand{\hmin}{H_{\min}}
\newcommand{\Dom}{\operatorname{Dom}}

\newcommand{\RE}{\operatorname{Re}}
\newcommand{\Tr}{\operatorname{Tr}}

\title{Self-adjoint extensions of differential operators on Riemannian manifolds}
\author{Ognjen Milatovic, Fran{\c c}oise Truc}
\address{Department of Mathematics
and Statistics \\ University of North Florida \\ Jacksonville, FL
32224 \\ USA.}
\email{omilatov@unf.edu}

\address{Grenoble University\\ Institut Fourier\\
Unit{\'e} mixte
 de recherche CNRS-UJF 5582\\
 BP 74, 38402-Saint Martin d'H\`eres Cedex, France.}
\email{francoise.truc@ujf-grenoble.fr}

\begin{document}
\maketitle

\begin{abstract} We study $H=D^*D+V$, where $D$ is a first order elliptic differential operator acting
on sections of a Hermitian vector bundle over a Riemannian manifold $M$, and $V$ is a Hermitian bundle endomorphism. 
In the case when $M$ is geodesically complete, we establish the essential self-adjointness of positive integer powers of $H$.
 In the case when $M$ is not necessarily geodesically complete, we give a sufficient condition for the essential self-adjointness of $H$, expressed in terms of the behavior of $V$ relative to the Cauchy boundary of $M$.
\end{abstract}
\section{Introduction}\label{S:literature}
As a fundamental problem in mathematical physics, self-adjointness of Schr\"odinger operators
has attracted the attention of researchers over many years now, resulting in numerous sufficient conditions for this property in $L^2(\mathbb{R}^{n})$. For reviews of the corresponding results, see, for instance, the books~\cite{CFKS,rs}.

The study of the corresponding problem in the context of a  non-compact Riemannian manifold was
initiated by Gaffney~\cite{ga,ga-55} with the proof of
the essential self-adjointness of the Laplacian on differential forms. About two decades later, Cordes (see Theorem 3 in ~\cite{Cordes}) proved the essential self-adjointness of positive integer powers of the operator
\begin{equation}\label{E:cordes-lap}
\Delta_{M,\mu}:= \ - \ \frac{1}{\kappa}\left(\frac{\partial}{\partial x^{i}}\left(\kappa g^{ij}\frac{\partial}{\partial x^{j}}\right)\right)
\end{equation}
on an $n$-dimensional geodesically complete Riemannian manifold $M$ equipped with a (smooth) metric $g=(g_{ij})$ (here, $(g^{ij})=((g_{ij})^{-1})$) and a positive smooth measure $d\mu$ (i.e.~in any local coordinates $x^{1},\,x^{2},\dots,x^{n}$ there exists a strictly
positive $C^{\infty}$-density $\kappa(x)$ such that $d\mu=\kappa(x)\,dx^{1}dx^{2}\dots dx^{n}$). Theorem~\ref{T:main-2} of our
paper extends this result to the operator $(D^*D+V)^{k}$ for all $k\in\mathbb{Z}_{+}$, where $D$ is a first order elliptic differential operator acting
on sections of a Hermitian vector bundle over a geodesically complete Riemannian manifold, $D^{*}$ is the formal adjoint of $D$, and $V$ is a self-adjoint Hermitian bundle endomorphism; see Section~\ref{SS:schr-op-def} for details.

In the context of a general Riemannian manifold (not necessarily geodesically complete), Cordes (see Theorem IV.1.1 in \cite{Cordes2} and Theorem 4 in \cite{Cordes}) proved the essential self-adjointness of $P^{k}$ for all $k\in\mathbb{Z}_{+}$, where
\begin{equation}\label{E:cordes-schro}
Pu:=\Delta_{M,\mu}u+qu, \qquad u\in C^{\infty}(M),
\end{equation}
and $q\in C^{\infty}(M)$ is real-valued. Thanks to a Roelcke-type estimate (see Lemma~\ref{L:roelcke} below), the technique of Cordes~\cite{Cordes2} can  be applied to the operator $(D^*D+V)^{k}$ acting
on sections of Hermitian vector bundles over a general Riemannian manifold. To make our exposition shorter, in Theorem~\ref{T:main-2} we consider the geodesically complete case. Our Theorem~\ref{T:main-4} concerns $(\nabla^*\nabla+V)^{k}$, where $\nabla$ is a metric connection on a Hermitian vector bundle over a non-compact geodesically complete Riemannian manifold. This result extends Theorem 1.1 of~\cite{Cordes3} where Cordes showed that if $(M,g)$ is non-compact and geodesically complete and $P$ is semi-bounded from below on $\mcomp$, then $P^{k}$ is essentially self-adjoint on $\mcomp$, for all $k\in\mathbb{Z}_{+}$.

For the remainder of the introduction, the notation $D^*D+V$ is used in the same sense as described earlier in this section. 
In the setting of geodesically complete Riemannian manifolds, the essential self-adjointness of $D^*D+V$ with $V\in L^{\infty}_{\loc}$ 
was established in~\cite{lm}, providing a generalization of the results in~\cite{br,ol,Oleinik94,sh1} concerning Schr\"odinger operators
 on functions (or differential forms). Subsequently, the operator $D^*D+V$ with a singular potential $V$ was considered in~\cite{B-M-S}. 
Recently, in the case $V\in L^{\infty}_{\loc}$, the authors of~\cite{B-C} extended the main result of~\cite{B-M-S} to the operator $D^*D+V$ acting on sections of infinite-dimensional bundles whose fibers are modules of finite type over a von Neumann algebra.

In the context of an incomplete Riemannian manifold, the authors of~\cite{G-M, Masamune-99, Masamune-05} studied the so-called Gaffney 
Laplacian, a self-adjoint realization of the scalar Laplacian generally different from the closure of $\Delta_{M,d \mu}|_{\mcomp}$. 
For a study of Gaffney Laplacian on differential forms, see~\cite{Masamune-07}.

Our Theorem~\ref{T:main-1} gives a condition on the behavior of $V$ relative to the Cauchy boundary of $M$ that will guarantee 
the essential self-adjointness of $D^*D+V$; for details see Section~\ref{SS:formulations} below. 
Related results can be found   in ~\cite{Brusentsev,Nen, Nen-11} in the context of (magnetic) Schr\"odinger operators
 on domains  in~${\mathbb{R}}^n$ , and in~\cite{Col-Tr} concerning the magnetic Laplacian 
on domains in ${\mathbb{R}}^n$ and certain types of Riemannian manifolds.

Finally, let us mention that Chernoff \cite{Chernoff} used
the hyperbolic equation approach to establish
the essential self-adjointness of positive integer powers of Laplace--Beltrami operator on
differential forms. This approach  was also applied in ~\cite{bandara,Chernoff77,Chumak,GK,GP,Shubin-92} to prove 
essential self-adjointness of second-order operators (acting on scalar functions or sections of Hermitian vector bundles) on Riemannian 
manifolds. Additionally, the authors of~\cite{GK,GP} used path integral techniques.

The paper is organized as follows. The main results are stated in Section~\ref{S:main},
a preliminary lemma is proven 
in Section~\ref{S:main-roelcke-1}, and the main results are proven in Sections~\ref{S:main-2-k-greater-than-1}--\ref{S:proof-main-1}.

\section{Main Results}\label{S:main}
\subsection{The setting}\label{SS:setting}
Let $M$ be an $n$-dimensional smooth, connected Riemannian manifold without boundary. We denote the Riemannian metric on $M$ by $g^{TM}$. We assume that $M$ is equipped with a positive smooth measure $d\mu$, i.e. in any local coordinates $x^{1},
\,x^{2},\dots,x^{n}$ there exists a strictly positive $C^{\infty}$-density $\kappa(x)$ such that $d\mu=\kappa(x)\,dx^{1}dx^{2}\dots dx^{n}$.
Let $E$ be a Hermitian vector bundle over $M$ and let $L^2(E)$ denote the Hilbert space of square integrable sections of $E$ with respect to the inner product
\begin{equation}\label{EE:inner}
       (u,v) \ = \
           \int_{M}\, \langle u(x),v(x)\rangle_{E_{x}}\, d\mu(x),
\end{equation}
where $\langle\cdot,\cdot\rangle_{E_{x}}$ is the fiberwise inner product. The corresponding norm in $L^2(E)$ is denoted by $\|\cdot\|$. In Sobolev space notations $W^{k,2}_{\loc}(E)$ used in this paper, the superscript $k\in\mathbb{Z}_{+}$ indicates the order of the highest derivative. The corresponding dual space is denoted by $W^{-k,2}_{\loc}(E)$.

Let $F$ be another Hermitian vector bundle on $M$. We consider a first order
differential operator $D\colon \ecomp\to\fcomp$, where $C^\infty_c$ stands for the space of
smooth compactly supported sections. In the sequel, by $\sigma(D)$ we denote the principal symbol of $D$.

\smallskip

\noindent\textbf{Assumption (A0)} Assume that $D$ is elliptic. Additionally, assume that there exists a constant $\lambda_0>0$ such that
\begin{equation}\label{E:equalst}
     |\sigma(D)(x,\xi)| \ \leq \ \lambda_0|\xi|,\qquad \textrm{for all }x\in M, \, \xi\in T_{x}^{*}M,
\end{equation}
where $|\xi|$ is the length of $\xi$ induced by the metric $g^{TM}$ and  $|\sigma(D)(x,\xi)|$ is the operator norm of $\sigma(D)(x,\xi)\colon E_x\to F_x$.

\begin{rem}\label{R:Laplacian} Assumption (A0) is satisfied if
$D=\nabla$, where $\nabla\colon C^\infty(E)\to C^{\infty}(T^*M\otimes E)$ is a covariant derivative corresponding to a metric
connection on a Hermitian vector bundle $E$ over $M$.
\end{rem}

\subsection{Schr\"odinger-type Operator}\label{SS:schr-op-def} Let $D^*\colon\fcomp\to\ecomp$ be the formal adjoint of $D$ with respect to
the inner product~(\ref{EE:inner}). We consider the operator
\begin{equation}\label{E:HV}
       H \ = \ D^*D \ + \ V,
\end{equation}
where $V\in L^{\infty}_{\loc}(\End{E})$ is a linear self-adjoint bundle
endomorphism. In other words, for all $x\in M$, the operator $V(x)\colon E_{x}\to E_{x}$ is self-adjoint and $|V(x)|\in L^{\infty}_{\loc}(M)$, where $|V(x)|$ is the norm of the operator $V(x)\colon E_{x}\to E_{x}$.

\subsection{Statements of Results}\label{SS:formulations}

\begin{thm}\label{T:main-2}   Let $M$, $g^{TM}$, and $d\mu$ be as in Section~\ref{SS:setting}. Assume that $(M,g^{TM})$ is geodesically complete. Let $E$ and $F$ be Hermitian vector bundles over $M$, and let $D\colon \ecomp\to \fcomp$ be a first order differential operator satisfying the assumption (A0).  Assume that $V\in C^{\infty}(\End{E})$ and
\[
V(x)\geq C,\qquad \textrm{for all }x\in M,
\]
where $C$ is a constant, and the inequality is understood in operator sense. Then $H^{k}$ is essentially self-adjoint on $\ecomp$, for all $k\in\mathbb{Z}_{+}$.
\end{thm}

\begin{rem}\label{R:corollary-ch} In the case $V=0$, the following result related to Theorem~\ref{T:main-2} can be deduced from~Chernoff (see Theorem 2.2 in~\cite{Chernoff}):

\smallskip

\emph{Assume that $(M,g)$ is a geodesically complete Riemannian manifold with metric $g$. Let $D$ be as in Theorem~\ref{T:main-2}, and define \/
    \[
       c(x) \ := \ \sup\{|\sigma(D)(x,\xi)|\colon\, |\xi|_{T_{x}^{*}M}=1\}.
    \]
Fix $x_0\in M$ and define
\[
c(r):=\sup_{x\in B(x_0,r)}c(x),
\]
where $r>0$ and $B(x_0,r):=\{x\in M\colon d_{g}(x_0,x)<r\}$. Assume that
\begin{equation}\label{E:divergence-c}
\int_{0}^{\infty}\frac{1}{c(r)}\,dr=\infty.
\end{equation}
Then the operator $(D^*D)^{k}$ is essentially self-adjoint on $\ecomp$ for all $k\in\mathbb{Z}_{+}$.}

At the end of this section we give an example of an operator for which Theorem~\ref{T:main-2} guarantees the essential self-adjointness
of $(D^*D)^{k}$, whereas Chernoff's result cannot be applied.
\end{rem}

The next theorem is concerned with operators whose potential $V$ is not necessarily semi-bounded from below.

\begin{thm}\label{T:main-4} Let $M$, $g^{TM}$, and $d\mu$ be as in Section~\ref{SS:setting}. Assume that $(M,g^{TM})$ is noncompact and geodesically complete. Let $E$ be a Hermitian vector bundle over $M$ and let $\nabla$ be a Hermitian connection on $E$.
Assume that $V\in C^{\infty}(\End{E})$ and
\begin{equation}\label{E:assumption-V-below-q}
V(x)\geq q(x),\qquad\textrm{for all }x\in M,
\end{equation}
where $q\in C^{\infty}(M)$ and the inequality is understood in the sense of operators $E_x\to E_x$.
Additionally, assume that
\begin{equation}\label{E:lap-semi-bounded}
((\Delta_{M,\mu}+q)u,u)\geq C\|u\|^2,\qquad\textrm{for all }u\in\mcomp,
\end{equation}
where $C\in\mathbb{R}$ and $\Delta_{M,\mu}$ is as in~(\ref{E:cordes-lap}) with $g$ replaced by $g^{TM}$.
Then the operator $(\nabla^*\nabla+V)^{k}$ is essentially self-adjoint on $\ecomp$, for all $k\in\mathbb{Z}_{+}$.
\end{thm}
\begin{rem} Let us stress that non-compactness is required in the proof to ensure the existence
of a positive smooth solution of an equation involving $\Delta_{M,\mu}+q$. In the case of a
 compact manifold, such a solution exists under an additional assumption; see Theorem III.6.3 in~\cite{Cordes2}.
\end{rem}

In our last result we will need the notion of Cauchy boundary. Let $d_{g^{TM}}$ be the distance function corresponding to the metric $g^{TM}$.
Let $(\widehat{M}, \widehat{d}_{g^{TM}})$ be the metric completion of $(M, d_{g^{TM}})$.
We define the \emph{Cauchy boundary} $\partial_{C}M$ as follows: $\partial_{C}M:=\widehat{M}\backslash M$.
Note that $(M,d_{g^{TM}})$ is metrically complete if and only if $\partial_{C}M$ is empty.
For $x\in M$ we define
\begin{equation}\label{E:dist-boundary}
r(x):=\inf_{z\in \partial_{C}M}\widehat{d}_{g^{TM}}(x,z).
\end{equation}

 We will also need  the following assumption:

\smallskip

\noindent\textbf{Assumption (A1)}
Assume that $\widehat{M}$ is a smooth manifold and that the metric $g^{TM}$ extends to $\partial_{C}M$.

\begin{rem} Let $N$ be a (smooth) $n$-dimensional Riemannian manifold without boundary. Denote the metric on $N$ by $g^{TN}$ and assume that $(N,g^{TN})$ is geodesically complete. Let $\Sigma$ be a $k$-dimensional closed sub-manifold of $N$ with $k<n$. Then $M:=N\backslash \Sigma$ has the properties $\widehat{M}=N$ and $\partial_{C}M=\Sigma$. Thus, assumption (A1) is satisfied.
\end{rem}

\begin{thm}\label{T:main-1}  Let $M$, $g^{TM}$, and $d\mu$ be as in Section~\ref{SS:setting}. Assume that (A1) is satisfied. Let $E$ and $F$ be Hermitian vector bundles over $M$, and let $D\colon \ecomp\to \fcomp$ be a first order differential operator satisfying the assumption (A0). Assume that $V\in L^{\infty}_{\loc}(\End{E})$ and there exists a constant $C$ such that
\begin{equation}\label{E:potential-minorant}
V(x) \geq\left(\frac{\lambda_0}{r(x)}\right)^2-C,\quad \textrm{ for all }x\in M,
\end{equation}
where  $\lambda_0$ is as in~(\ref{E:equalst}), the distance $r(x)$ is as in~(\ref{E:dist-boundary}), and the inequality is understood in the sense of linear operators $E_{x}\to E_{x}$. Then $H$ is essentially self-adjoint on $\ecomp$.
\end{thm}
In order to describe the example mentioned in Remark~\ref{R:corollary-ch}, we need the following
\begin{rem}\label{SS:Fmetric}
As explained in~\cite{B-M-S}, we can use a first-order elliptic operator $D\colon\ecomp\to \fcomp$ to define a metric on $M$.  For
$\xi,\eta\in T^*_xM$, define
\begin{equation}\label{E:trace}
       \langle \xi,\eta \rangle \ = \
       \frac{1}{m}\, \RE\, \Tr\left(\left(\sigma(D)(x,\xi)\right)^*\sigma(D)(x,\eta)\right),
       \qquad m= \dim E_x,
\end{equation}
where $\Tr$ denotes the usual trace of a linear operator.  Since $D$ is an elliptic
first-order differential operator and $\sigma(D)(x,\xi)$ is linear in $\xi$, it is easily checked that~(\ref{E:trace}) defines an inner product on $T^*_xM$.  Its dual defines a Riemannian metric on $M$. Denoting this metric by $g^{TM}$ and using elementary linear
algebra, it follows that~(\ref{E:equalst}) is satisfied with $\lambda_0=\sqrt{m}$.
\end{rem}
\begin{ex} \emph{Let $M=\mathbb{R}^{2}$ with the standard metric and measure, and $V=0$. Denoting respectively by $C_{c}^{\infty}(\mathbb{R}^{2};\mathbb{R})$
and $C_{c}^{\infty}(\mathbb{R}^{2};\mathbb{R}^2)$  the spaces of smooth compactly supported functions
$f\colon \mathbb{R}^{2}\to \mathbb{R}$ and $f\colon \mathbb{R}^{2}\to \mathbb{R}^2$, we define the  operator
   $D\colon C_{c}^{\infty}(\mathbb{R}^{2};\mathbb{R})\to C_{c}^{\infty}(\mathbb{R}^{2};\mathbb{R}^2)$ by
\[
      D \ = \ \left(\begin{array}{c}a(x,y)\frac{\partial}{\partial x}\\
                          b(x,y)\frac{\partial}{\partial y}\\
                      \end{array}\right),
\]
where
\begin{eqnarray}
     a(x,y) \ &=& \ (1-\cos(2\pi e^{x}))x^{2}+1;\notag\\
     b(x,y) \ &=& \ (1-\sin(2\pi e^{y}))y^{2}+1.\notag
\end{eqnarray}
Since $a,b$ are smooth real-valued nowhere vanishing functions in $\mathbb{R}^2$, it follows that the operator $D$ is elliptic.
We are interested in the operator
\[
      H \ := \ D^*D
      \ = \
      -\frac{\partial}{\partial x}
        \left(a^{2}\frac{\partial}{\partial x}\right)-
           \frac{\partial}{\partial y}
            \left(b^{2}\frac{\partial}{\partial y}\right).
\]
The matrix of the inner product on $T^*M$ defined by $D$ via (\ref{E:trace}) is ${\rm diag}(a^2/2,b^2/2)$.
The matrix  of the corresponding Riemannian metric $g^{TM}$ on $M$ is
${\rm diag}(2a^{-2},2b^{-2})$, so the metric itself is $ds^2=2a^{-2}dx^2+2b^{-2}dy^2$ and it
 is geodesically complete (see Example 3.1 of~\cite{B-M-S}). Moreover, thanks to Remark~\ref{SS:Fmetric}, assumption  (A0)  is
satisfied. Thus, by Theorem~\ref{T:main-2} the operator $(D^*D)^{k}$ is essentially self-adjoint for all $k\in\mathbb{Z}_{+}$.
Furthermore, in Example 3.1 of~\cite{B-M-S} it was shown that for the considered operator $D$ the condition~(\ref{E:divergence-c}) is not satisfied.
Thus, the result stated in Remark~\ref{R:corollary-ch} does not apply.}
\end{ex}

\section{Roelcke-type Inequality}\label{S:main-roelcke-1}
Let $M$, $d\mu$, $D$, and $\sigma(D)$ be as in Section~\ref{SS:setting}. Set $\widehat{D}:=-i\sigma(D)$, where $i=\sqrt{-1}$. Then for any Lipschitz function $\psi\colon M\to \mathbb{R}$ and $u\in W^{1,2}_{\loc}(E)$ we have
\begin{equation}\label{E:defsym}
       D(\psi u) \ = \ \symd(d\psi)u+\psi Du,
\end{equation}
where we have suppressed $x$ for simplicity. We also note that $\symdt(\xi)=-(\symd(\xi))^{*}$, for all $\xi\in
T_{x}^{*}M$.

For a compact set $K\subset M$, and $u,\, v\in W^{1,2}_{\loc}(E)$, we define
\begin{equation}\label{E:inner-compact}
(u,v)_{K}:=\int_{K}\langle u(x),v(x)\rangle\,d\mu(x), \qquad (Du,Dv)_{K}:=\int_{K}\langle Du(x),Dv(x)\rangle\,d\mu(x).
\end{equation}
In order to prove Theorem~\ref{T:main-2} we need the following important lemma, which  is an extension of Lemma 2.1 in \cite{Cordes2} to operator~(\ref{E:HV}). In the context of the scalar Laplacian on a Riemannian manifold, this kind of result is originally due to Roelcke~\cite{Roelcke}.

\begin{lem}\label{L:roelcke}  Let $M$, $g^{TM}$, and $d\mu$ be as in Section~\ref{SS:setting}. Let $E$ and $F$ be Hermitian vector bundles over $M$, and let $D\colon \ecomp\to \fcomp$ be a first order differential operator satisfying the assumption (A0). Let $\rho\colon M\to[0,\infty)$ be a function satisfying the following properties:
\begin{enumerate}
\item [(i)] $\rho(x)$ is Lipschitz continuous with respect to the distance induced by  the metric $g^{TM}$;
\item [(ii)] $\rho(x_0)=0$, for some fixed $x_0\in M$;
\item [(iii)] the set $B_{T}:=\{x\in M\colon \rho(x)\leq T\}$ is compact, for some $T>0$.
\end{enumerate}
Then the following inequality holds for all $u\in W^{2,2}_{\loc}(E)$ and $v\in W^{2,2}_{\loc}(E)$:
\begin{equation}\label{E:roelcke}
\int_{0}^{T} |(Du,Dv)_{B_t}-(D^*Du,v)_{B_t}|\,dt\leq \lambda_0\int_{B_{T}}|d\rho(x)||Du(x)||v(x)|\,d\mu(x),
\end{equation}
where $B_{t}$ is as in (iii)  (with $t$ instead of $T$), the constant $\lambda_0$ is as in~(\ref{E:equalst}), and $|d\rho(x)|$ is the length of $d\rho(x)\in T_{x}^*M$ induced by $g^{TM}$.
\end{lem}
\begin{proof} For $\varepsilon>0$ and $t\in (0,T)$, we define a continuous piecewise linear function $F_{\varepsilon,t}$ as follows:
\[
F_{\varepsilon,t}(s)= \left\{
\begin{array}{l}
1  {\rm ~ for~}  s<t-\varepsilon   \\

(t-s)/{\varepsilon} {\rm ~ for~} t-\varepsilon\leq s <t \\
0   {\rm ~ for~ }  s \geq t
\end{array}
\right.
\]
The function  $f_{\varepsilon,t}(x):=F_{\varepsilon,t}(\rho(x))$,  is  Lipschitz continuous with respect to the distance induced
by the metric $g^{TM}$, and $d(f_{\varepsilon,t}(\rho(x)))=(F'_{\varepsilon,t}(\rho(x)))d\rho(x)$. Moreover we have
 $f_{\varepsilon,t}v\in W^{1,2}_{\loc}(E)$ for all $v\in W^{1,2}_{\loc}(E)$, since
\[
D(f_{\varepsilon,t}v)=\symd(df_{\varepsilon,t})v+f_{\varepsilon,t}Dv.
\]
 It follows from the compactness of  $B_{T}$ that $B_{t}$ is compact for all $t\in (0,T)$.  Using integration by parts (see Lemma 8.8 in~\cite{B-M-S}), for all $u\in W^{2,2}_{\loc}(E)$ and $v\in W^{2,2}_{\loc}(E)$ we have
\begin{equation}\nonumber
(D^*Du, vf_{\varepsilon,t})_{B_{t}} = (Du, D(vf_{\varepsilon,t}))_{B_{t}} = (Du,f_{\varepsilon,t}Dv)_{B_{t}} + (Du,\symd(df_{\varepsilon,t})v)_{B_{t}},
\end{equation}
which, together with~(\ref{E:equalst}), gives
\begin{align}\label{E:roelcke-integral}
&|(Du,f_{\varepsilon,t}Dv)_{B_{t}}-(D^*Du, vf_{\varepsilon,t})_{B_{t}}|=|(Du,\symd(df_{\varepsilon,t})v)_{B_{t}}|\nonumber\\
&\leq \int_{B_{t}}|Du(x)||\symd(df_{\varepsilon,t}(x))v(x)|\,d\mu(x)
\leq \lambda_0 \int_{B_{t}}|Du(x)||df_{\varepsilon,t}(x)||v(x)|\,d\mu(x)\nonumber\\
&= \lambda_0\int_{B_{t}}|Du(x)||F'_{\varepsilon,t}(\rho(x))||d\rho(x)||v(x)|\,d\mu(x)\nonumber\\
&\leq \lambda_0\int_{B_{T}}|Du(x)||F'_{\varepsilon,t}(\rho(x))||d\rho(x)||v(x)|\,d\mu(x),
\end{align}
where  $|df_{\varepsilon,t}(x)|$ and $|d\rho(x)|$ are the norms of $df_{\varepsilon,t}(x)\in T_{x}^*M$ and $d\rho(x)\in T_{x}^*M$ induced by $g^{TM}$.

Fixing $\varepsilon>0$, integrating the leftmost and the rightmost side of~(\ref{E:roelcke-integral}) from $t=0$ to $t=T$, and noting that $F'_{\varepsilon,t}(\rho(x))$ is the only term on the rightmost side depending on $t$, we obtain
\begin{align}\label{E:roelcke-integral-1}
&\int_{0}^{T}|(Du,f_{\varepsilon,t}Dv)_{B_{t}}-(D^*Du, vf_{\varepsilon,t})_{B_{t}}|\,dt\nonumber\\
&\leq \lambda_0\int_{B_{T}}|Du(x)||d\rho(x)||v(x)|I_{\varepsilon}(x)\,d\mu(x),
\end{align}
where
\[
I_{\varepsilon}(x):=\int_{0}^{T}|F'_{\varepsilon,t}(\rho(x))|\,dt.
\]
We now let $\varepsilon\to 0+$ in~(\ref{E:roelcke-integral-1}). On the left-hand side of~(\ref{E:roelcke-integral-1}), as $\varepsilon\to 0+$, we have $f_{\varepsilon,t}(x)\to\chi_{B_t}(x)$ almost everywhere, where $\chi_{B_t}(x)$ is the characteristic function of the set $B_t$. Additionally, $|f_{\varepsilon,t}(x)|\leq 1$ for all $x\in B_t$ and all $t\in (0,T)$; thus, by dominated convergence theorem, as $\varepsilon\to 0+$ the left-hand side of~(\ref{E:roelcke-integral-1}) converges to the left-hand side of~(\ref{E:roelcke}). On the right-hand side of~(\ref{E:roelcke-integral-1}) an easy calculation shows that $I_{\varepsilon}(x)\to 1$, as $\varepsilon\to 0+$. Additionally, we have  $|I_{\varepsilon}(x)|\leq 1$, a.e. on $B_{T}$; hence, by the dominated convergence theorem, as $\varepsilon\to 0+$ the right-hand side of~(\ref{E:roelcke-integral-1}) converges to the right-hand side of~(\ref{E:roelcke}). This establishes the inequality~(\ref{E:roelcke}).
\end{proof}

\section{Proof of Theorem~\ref{T:main-2}}\label{S:main-2-k-greater-than-1}

We first give the definitions of minimal and maximal operators associated with the expression $H$ in~(\ref{E:HV}).

\subsection{Minimal and Maximal Operators}\label{SS:mm-1} We define $\hmin u:=Hu$, with $\Dom(\hmin):=\ecomp$,
 and $\hmax:=(\hmin)^{*}$, where $T^*$ denotes the adjoint of operator $T$.
Denoting $\mathscr{D}_{\max}:=\{u\in L^2(E)\colon Hu\in L^2(E)\}$, we recall the following well-known property:
 $\Dom(\hmax)=\mathscr{D}_{\max}$ and $\hmax u=Hu$ for all $u\in\mathscr{D}_{\max}$.

From now on, throughout this section, we assume that the hypotheses of Theorem~\ref{T:main-2} are satisfied.
Let $x_0\in M$, and define $\rho(x):=d_{g^{TM}}(x_0,x)$, where $d_{g^{TM}}$ is the distance function corresponding to the metric $g^{TM}$.
 By the definition of $\rho(x)$ and the geodesic completeness of $(M, g^{TM})$, it follows that $\rho(x)$ satisfies all hypotheses of
 Lemma~\ref{L:roelcke}. Using Lemma~\ref{L:roelcke} and Proposition~\ref{P:abstract-power} below, we are able to apply the method of
Cordes~\cite{Cordes,Cordes2}  to our context. As we will see, Cordes's technique reduces our problem to a system of ordinary differential
 inequalities of the same type as in Section IV.3 of \cite{Cordes2}.

\begin{Prop}\label{P:abstract-power} Let $A$ be a densely defined operator with domain $\mathscr{D}$ in a Hilbert space $\mathscr{H}$.
Assume that $A$ is semi-bounded from below, that $A\mathscr{D}\subseteq \mathscr{D}$, and  that there exists $c_0\in\mathbb{R}$ such that the following two properties hold:
\begin{enumerate}
\item [(i)] $((A+c_0I)u,u)_{\mathscr{H}}\geq \|u\|_{\mathscr{H}}^2$, for all $u\in\mathscr{D}$, where $I$ denotes the identity operator in $\mathscr{H}$;

\item [(ii)] $(A+c_0I)^{k}$ is essentially self-adjoint on $\mathscr{D}$, for some $k\in\mathbb{Z}_{+}$.
\end{enumerate}
Then, $(A+cI)^{j}$ is essentially self-adjoint on $\mathscr{D}$, for all $j=1,2,\dots, k$ and all $c\in\mathbb{R}$.
\end{Prop}
\begin{rem} To prove Proposition~\ref{P:abstract-power}, one may mimick the proof of Proposition 1.4 in \cite{Cordes2},
which was carried out for  the operator $P$ defined in~(\ref{E:cordes-schro}) with $\mathscr{D}=\mcomp$, since
 only abstract functional analysis facts and the property $P\mathscr{D}\subseteq\mathscr{D}$ were used.
\end{rem}

We start the proof of Theorem~\ref{T:main-2} by noticing  that the operator $\hmin$ is essentially self-adjoint on $\ecomp$; see Corollary 2.9 in \cite{B-M-S}.
Thanks to Proposition~\ref{P:abstract-power}, whithout any loss of generality we can change $V(x)$ to $ V(x) + C\,\textrm{Id}(x)$ ,
where $C$ is a sufficiently large constant in order to have
\begin{equation}\label{E:V-by-one}
V(x)\geq \,(\lambda_0^2+1)\textrm{Id}(x),\qquad\textrm{for all }x\in M,
\end{equation}
where $\lambda_0$ is as in~(\ref{E:equalst}) and $\textrm{Id}(x)$ is the identity endomorphism of $E_{x}$.
 Using non-negativity of $D^*D$ and~(\ref{E:V-by-one}) we have
\begin{equation}\label{E:semib-set-up-1}
(\hmin u,u)\geq\|u\|^2,\qquad\textrm{for all }u\in\ecomp,
\end{equation}
which leads to
\[
\|u\|^2\leq (Hu,u)\leq \|Hu\|\|u\|,\qquad\textrm{for all }u\in\ecomp,
\]
and, hence, $\|Hu\|\geq \|u\|$, for all $u\in\ecomp$. Therefore,
\begin{equation}\label{E:H-2-power-s}
(H^2u,u)=(Hu,Hu)=\|Hu\|^2\geq \|u\|^2,\qquad\textrm{for all }u\in\ecomp,
\end{equation}
and
\[
(H^3u,u)=(HHu,Hu)\geq \|Hu\|^2\geq \|u\|^2,\qquad\textrm{for all }u\in\ecomp.
\]
By~(\ref{E:H-2-power-s}) we have
\[
\|u\|^2\leq (H^2u,u)\leq \|H^2u\|\|u\|, \qquad\textrm{for all }u\in\ecomp,
\]
and, hence, $\|H^2u\|\geq \|u\|$, for all $u\in\ecomp$. This, in turn, leads to
\[
(H^4u,u)=(H^2u,H^2u)=\|H^2u\|^2\geq \|u\|^2,\qquad\textrm{for all }u\in\ecomp.
\]
Continuing like this, we obtain $(H^ku,u)\geq \|u\|^2$, for all $u\in\ecomp$ and all $k\in\mathbb{Z}_{+}$.
In this case, by an abstract fact (see Theorem X.26 in~\cite{rs}), the essential self-adjointness of $H^{k}$ on $\ecomp$ is equivalent to the following statement: if $u\in L^2(E)$ satisfies $H^{k}u=0$, then $u=0$.

Let $u\in L^2(E)$ satisfy $H^{k}u=0$. Since $V\in C^{\infty}(E)$, by local elliptic regularity it follows that $u\in C^{\infty}(E)\cap L^2(E)$. Define
\begin{equation}\label{E:def-f-j}
f_{j}:=H^{k-j}u,\qquad j=0,\pm 1, \pm 2, \dots
\end{equation}
Here, in the case $k-j<0$, the definition~(\ref{E:def-f-j}) is interpreted as $((\hmax)^{-1})^{j-k}$. We already noted that $\hmin$ is
essentially self-adjoint and positive. Furthermore, it is well known that the self-adjoint closure of $\hmin$ coincides with $\hmax$.
Therefore $\hmax$ is a positive self-adjoint operator, and  $(\hmax)^{-1}\colon L^2(E)\to L^2(E)$ is bounded.
This, together with $f_k=u\in L^2(E)$ explains the following property: $f_j\in L^2(E)$, for all $j\geq k$.
Additionally, observe that $f_j= 0$ for all $j\leq 0$ because $f_0= 0$. Furthermore, we note that $f_j\in C^{\infty}(E)$, for all $j\in\mathbb{Z}$.
The last assertion is obvious for $j\leq k$, and for $j>k$ it can be seen by showing that $H^{j}f_j=0$ in distributional sense and using $f_j\in L^2(E)$
together with local elliptic regularity. To see this, let $v\in \ecomp$ be arbitrary, and note that
\[
(f_j,H^{j}v)=(H^{k-j}u,H^{j}v)=(u,H^{k}v)=(H^ku,v)=0.
\]
Finally, observe that
\begin{equation}\label{E:obs-1}
H^{l}f_j=f_{j-l},\qquad\textrm{for all }j\in \mathbb{Z}\textrm{ and } l\in\mathbb{Z}_{+}\cup\{0\}.
\end{equation}

With $f_j$ as in~(\ref{E:def-f-j}), define the functions $\alpha_j$ and $\beta_j$ on the interval $0\leq T<\infty$ by the formulas
\begin{equation}\label{E:alpha-beta-j}
\alpha_j(T):=\lambda_0^2\int_{0}^{T}(f_j,f_j)_{B_{t}}\,dt,\qquad \beta_j(T):=\int_{0}^{T}(D f_j,D f_j)_{B_{t}}\,dt,
\end{equation}
where $\lambda_0$ is as in~(\ref{E:V-by-one}) and $(\cdot,\cdot)_{B_t}$ is as in~(\ref{E:inner-compact}).

In the sequel, to simplify the notations, the functions $\alpha_j(T)$ and $\beta_j(T)$, the inner products $(\cdot,\cdot)_{B_{t}}$, and the  corresponding norms
$\|\cdot\|_{B_t}$ appearing in~(\ref{E:alpha-beta-j}) will be denoted by $\alpha_j$, $\beta_j$, $(\cdot,\cdot)_{t}$,
and $\|\cdot\|_{t}$, respectively.

Note that $\alpha_j$ and $\beta_j$ are absolutely continuous on $[0,\infty)$. Furthermore, $\alpha_j$ and $\beta_j$ have a left first derivative and a right first derivative at each point. Additionally, $\alpha_j$ and $\beta_j$ are differentiable, except at (at most) countably many points. In the sequel, to simplify notations, we shall denote the right first derivatives of $\alpha_j$ and $\beta_j$ by $\alpha_j'$ and $\beta_j'$. Note that $\alpha_j$, $\beta_j$, $\alpha_j'$ and $\beta_j'$ are non-decreasing and non-negative functions. Note also that $\alpha_j$ and $\beta_j$ are convex functions. Furthermore, since $f_j=0$ for all $j\leq 0$, it follows that $\alpha_j\equiv 0$ and $\beta_j\equiv 0$ for all $j\leq 0$. Finally, using~(\ref{E:V-by-one}) and the property $f_j\in L^2(E)\cap C^{\infty}(E)$ for all $j\geq k$, observe that
\[
\lambda_0^2(f_j,f_j)+(Df_j,D f_j)\leq (Vf_j,f_j)+(D f_j,D f_j)=(f_j,Hf_j)=(f_j,f_{j-1})<\infty,
\]
for all $j>k$. Here, ``integration by parts" in the first equality is justified because $\hmin$ is essentially self-adjoint (i.e.~$\ecomp$ is an operator core of $\hmax$). Hence, $\alpha_j'$ and $\beta_j'$ are bounded for all $j>k$. It turns out that $\alpha_j$ and $\beta_j$ satisfy a system of differential inequalities,
as seen in the next proposition.

\begin{Prop}\label{P:system} Let $\alpha_j$ and $\beta_j$ be as in~(\ref{E:alpha-beta-j}). Then, for all $j\geq 1$ and all $T\geq 0$ we have
\begin{equation}\label{E:system-alpha-beta-j}
\alpha_j+\beta_j\leq \sqrt{\alpha_j'\beta_j'}+\sum_{l=0}^{\infty}\left(\sqrt{\alpha'_{j+l+1}\beta'_{j-l-1}}+\sqrt{\alpha'_{j-l-1}\beta'_{j+l+1}}\right)
\end{equation}
and
\begin{equation}\label{E:system-alpha-j}
\alpha_j\leq \lambda_0^2\left(\sum_{l=0}^{\infty}\left(\sqrt{\alpha'_{j+l+1}\beta'_{j-l}}+\sqrt{\alpha'_{j-l}\beta'_{j+l+1}}\right)\right),
\end{equation}
where $\lambda_0$ is as in~(\ref{E:V-by-one}) and $\alpha_i'$, $\beta_i'$ denote the right-hand derivatives.
\end{Prop}
\begin{rem} Note that the sums in~(\ref{E:system-alpha-beta-j}) and~(\ref{E:system-alpha-j}) are finite since $\alpha_{i}\equiv 0$ and $\beta_{i}\equiv 0$ for $i\leq 0$. As our goal is to show that $f_k=u=0$, we will only use the first $k$ inequalities in (\ref{E:system-alpha-beta-j}) and the first $k$ inequalities in~(\ref{E:system-alpha-j}).
\end{rem}

\noindent\textbf{Proof of Proposition~\ref{P:system}.} From~(\ref{E:alpha-beta-j}) and~(\ref{E:V-by-one}) it follows that
\begin{equation}\label{E:ab-1}
\alpha_j+\beta_j\leq \int_{0}^{T}\left((f_j, Vf_j)_{t}+(D f_j, D f_j)_{t}\right)\,dt.
\end{equation}

We start from~(\ref{E:ab-1}), use~(\ref{E:roelcke}), Cauchy--Schwarz inequality, and~(\ref{E:obs-1}) to obtain
\begin{align*}
\alpha_j+\beta_j&\leq \int_{0}^{T}\left((f_j, Vf_j)_{t}+(D f_j, D f_j)_{t}\right)\,dt\nonumber\\
&=\int_{0}^{T}\left|(f_j, Hf_j)_{t}-(f_j, D^*D f_j)_{t}+(D f_j, D f_j)_{t}\right|\,dt\nonumber\\
&\leq \lambda_0\int_{B_{T}}|D f_j(x)||f_j(x)|\,d\mu(x)+\int_{0}^{T}|(f_j, Hf_j)_t|\,dt\nonumber\\
&\leq \sqrt{\alpha_j'\beta_j'}+ \int_{0}^{T}|(Hf_{j+1}, f_{j-1})_t|\,dt.
\end{align*}

We continue the process as follows:
\begin{align*}
\alpha_j+\beta_j&\leq \sqrt{\alpha_j'\beta_j'}+\int_{0}^{T}|(Hf_{j+1}, f_{j-1})_t|\,dt\nonumber\\
&=\sqrt{\alpha_j'\beta_j'}+\int_{0}^{T}\left|(D^*D f_{j+1}, f_{j-1})_t+(f_{j+1}, Vf_{j-1})_{t}\right|\,dt\nonumber\\
&\leq \sqrt{\alpha_j'\beta_j'}+\int_{0}^{T}\left|(D^*D f_{j+1}, f_{j-1})_t-(D f_{j+1}, D f_{j-1})_{t}\right|\,dt\nonumber\\
&+\int_{0}^{T}\left|(D f_{j+1}, D f_{j-1})_{t}-(f_{j+1},D^*D f_{j-1})_{t}\right|\,dt +\int_{0}^{T}|(f_{j+1}, Hf_{j-1})_{t}|\,dt\nonumber\\
&\leq \sqrt{\alpha_j'\beta_j'} +\sqrt{\alpha_{j+1}'\beta_{j-1}'} +\sqrt{\alpha_{j-1}'\beta_{j+1}'}+\int_{0}^{T}|(Hf_{j+2}, f_{j-2})_{t}|\,dt,
\end{align*}
where we used triangle inequality,~(\ref{E:roelcke}), Cauchy--Schwarz inequality, and~(\ref{E:obs-1}).
We continue like this until the last term reaches the subscript $j-l\leq 0$, which makes the last term equal zero by properties of $f_{i}$ discussed above. This establishes~(\ref{E:system-alpha-beta-j}).

To show~(\ref{E:system-alpha-j}), we start from the definition of $\alpha_j$, use~(\ref{E:roelcke}), Cauchy--Schwarz inequality, and~(\ref{E:obs-1}) to obtain
\begin{align*}
\alpha_j&= \lambda_0^2\int_{0}^{T}(f_j,f_j)_{t}\,dt=\lambda_0^2\int_{0}^{T}|(f_j,Hf_{j+1})_{t}|\,dt\nonumber\\
&=\lambda_0^2\int_{0}^{T}|(f_j,D^*D f_{j+1})_{t}+(Vf_j,f_{j+1})_{t}|\,dt\nonumber\\
&\leq \lambda_0^2\int_{0}^{T}|(f_j,D^*D f_{j+1})_{t}-(D f_j,D f_{j+1})_{t}|\,dt\nonumber\\
&+\lambda_0^2\int_{0}^{T}|(D f_j,D f_{j+1})_{t}-(D^*D f_j,f_{j+1})_{t}|\,dt+\lambda_0^2\int_{0}^{T}|(Hf_j,f_{j+1})_{t}|\,dt\nonumber\\
&\leq \lambda_0^2\left(\sqrt{\alpha_{j+1}'\beta_{j}'} +\sqrt{\alpha_{j}'\beta_{j+1}'}\right)+\lambda_0^2\int_{0}^{T}|(f_{j-1},f_{j+1})_{t}|\,dt.
\end{align*}
We continue like this until the last term reaches the subscript $j-l\leq 0$, which makes the last term equal zero by properties of $f_{i}$ discussed above. This establishes~(\ref{E:system-alpha-j}). $\hfill\square$

\medskip

\noindent\textbf{End of the proof of Theorem~\ref{T:main-2}.} We will now transform the system~(\ref{E:system-alpha-beta-j})--(\ref{E:system-alpha-j}) by introducing new variables:
\begin{equation}\label{E:omega-j-theta-j}
\omega_j(T):=\alpha_j(T)+\beta_j(T),\qquad \theta_j(T):=\alpha_j(T)-\beta_j(T)\qquad T\in[0,\infty).
\end{equation}
To carry out the transformation, observe that Cauchy--Schwarz inequality applied to vectors $\langle\sqrt{\alpha'_i}, \sqrt{\beta'_i}\rangle$ and $\langle\sqrt{\beta'_p},\sqrt{\alpha'_p}\rangle$ in $\mathbb{R}^2$ gives
\[
\sqrt{\alpha_i'\beta_p'}+\sqrt{\alpha_p'\beta_i'}\leq \sqrt{\omega_i'\omega_p'},
\]
which, together with (\ref{E:system-alpha-beta-j})--(\ref{E:system-alpha-j}) leads to
\begin{equation}\label{E:system-omega-theta-j}
\omega_j\leq \frac{1}{2}\sqrt{(\omega_j')^2-(\theta_j')^2}+\sum_{l=0}^{\infty}\sqrt{\omega'_{j+l+1}\omega'_{j-l-1}}
\end{equation}
and
\begin{equation}\label{E:system-omega-j}
\frac{1}{2}(\omega_j+\theta_j)\leq \lambda_0^2\left(\sum_{l=0}^{\infty}\sqrt{\omega'_{j+l+1}\omega'_{j-l}}\right),
\end{equation}
where $\lambda_0$ is as in~(\ref{E:V-by-one}) and $\omega_i'$, $\theta_i'$ denote the right-hand derivatives.

The functions $\omega_j$ and $\theta_j$ satisfy the following properties: (i) $\omega_j$ and $\theta_j$ are absolutely continuous on $[0,\infty)$, and the right-hand derivatives $\omega_j'$ and $\theta_j'$ exist everywhere; (ii) $\omega_j$  and $\omega_j'$ are non-negative and non-increasing; (iii) $\omega_j$ is convex; (iv) $\omega_{j}'$ is bounded for all $j\geq k$; (v) $\omega_j(0)=\theta_j(0)=0$; and (vi) $|\theta_j(T)|\leq\omega_j(T)$ and $|\theta_j'(T)|\leq\omega_j'(T)$ for all $T\in[0,\infty)$.

In Section IV.3 of~\cite{Cordes2} it was shown that if $\omega_j$ and $\theta_j$ are functions satisfying the above described properties (i)--(vi) and the system (\ref{E:system-omega-theta-j})--(\ref{E:system-omega-j}), then $\omega_{j}\equiv 0$  for all $j=1,2,\dots, k$. In particular, we have $\omega_k(T)=0$, for all $T\in[0,\infty)$, and hence $f_k=0$. Going back to~(\ref{E:def-f-j}), we get $u=0$, and this concludes the proof of essential self-adjointness of $H^{k}$ on $\ecomp$. The essential self-adjointness of $H^2$, $H^{3}$, $\dots$, and $H^{k-1}$ on $\ecomp$ follows by Proposition~\ref{P:abstract-power}.
$\hfill\square$

\section{Proof of Theorem~\ref{T:main-4}} We adapt the proof of Theorem 1.1 in~\cite{Cordes3} to our type of operator. By assumption~(\ref{E:lap-semi-bounded}) it follows that
\begin{equation}\label{E:lap-semi-bounded-1}
((\Delta_{M,\mu}+q-C+1)u,u)\geq \|u\|^2,\qquad\textrm{for all }u\in\mcomp.
\end{equation}
Since~(\ref{E:lap-semi-bounded-1}) is satisfied and since $M$ is non-compact and $g^{TM}$ is geodesically complete, a result of Agmon~\cite{Agmon-83} (see also Proposition III.6.2 in~\cite{Cordes2})  guarantees the existence of a function $\gamma\in C^{\infty}(M)$ such that $\gamma(x)>0$ for all $x\in M$, and
\begin{equation}\label{E:gamma-solution}
(\Delta_{M,\mu}+q-C+1)\gamma=\gamma.
\end{equation}
We now use the function $\gamma$ to transform the operator $H=\nabla^*\nabla+V$. Let $L_{\mu_1}^2(E)$ be the space of square integrable sections of $E$ with inner product $(\cdot,\cdot)_{\mu_1}$ as in~(\ref{EE:inner}), where $d\mu$ is replaced by $d\mu_1:=\gamma^2 d\mu$. For clarity, we denote $L^2(E)$ from Section~\ref{SS:setting} by $L_{\mu}^2(E)$.  In what follows, the formal adjoints of $\nabla$ with respect to inner products $(\cdot,\cdot)_{\mu}$ and $(\cdot,\cdot)_{\mu_1}$ will be denoted by $\nabla^{*,\mu}$ and $\nabla^{*,\mu_1}$, respectively. It is easy to check that the map $T_{\gamma}\colon L_{\mu}^2(E)\to L_{\mu_1}^2(E)$ defined by $Tu:=\gamma^{-1} u$ is unitary. Furthermore, under the change of variables $u\mapsto\gamma^{-1} u$, the differential expression $H=\nabla^{*,\mu}\nabla+V$ gets transformed into $H_1:=\gamma^{-1}H\gamma$. Since $T$ is unitary, the essential self-adjointness of $H^k|_{\ecomp}$ in $L_{\mu}^2(E)$ is equivalent to essential self-adjointness of $(H_{1})^k|_{\ecomp}$ in $L_{\mu_1}^2(E)$.

In the sequel, we will show that $H_1$ has the following form:
\begin{equation}\label{E:op-H-1}
H_1=\nabla^{*,\mu_1}\nabla+\widetilde{V},
\end{equation}
with
\[
\widetilde{V}(x):=\frac{\Delta_{M,\mu}\gamma}{\gamma}\,\textrm{Id}(x)+V(x).
\]
To see this, let $w,\,z\in\ecomp$ and consider
\begin{align}\label{E:t-1}
&(H_1w,z)_{\mu_1}=\int_{M}\langle\gamma^{-1}H(\gamma w) ,z\rangle\,\gamma^2d\mu= \int_{M}\langle H(\gamma w) ,\gamma z\rangle\,d\mu=(H(\gamma w) ,\gamma z)_{\mu}\nonumber\\
&=(\nabla(\gamma w),\nabla (\gamma z))_{\mu}+(V\gamma w, \gamma z)_{\mu}=(\gamma^2\nabla w,\nabla z)_{\mu}+(d\gamma\otimes w,d\gamma\otimes z)_{L^2_{\mu}(T^*M\otimes E)}\nonumber\\
&+(\gamma\nabla w,d\gamma\otimes z)_{L^2_{\mu}(T^*M\otimes E)}+(d\gamma\otimes w,\gamma \nabla z)_{L^2_{\mu}(T^*M\otimes E)}+(V\gamma w, \gamma z)_{\mu}.
\end{align}
Setting $\xi:=d(\gamma^2/2)\in T^*M$ and using equation (1.34) in Appendix C of~\cite{Taylor} we have
\begin{align}\label{E:t-2}
&(\gamma\nabla w,d\gamma\otimes z)_{L^2_{\mu}(T^*M\otimes E)}=
(\nabla w, \xi \otimes z)_{L^2_{\mu}(T^*M\otimes E)}=(\nabla_{X} w, z)_{\mu},
\end{align}
where $X$ is the vector field associated with  $\xi\in T^*M$ via the metric $g^{TM}$.

Furthermore, by equation (1.35) in Appendix C of~\cite{Taylor} we have
\begin{align}\label{E:t-3}
&(d\gamma\otimes w,\gamma \nabla z)_{L^2_{\mu}(T^*M\otimes E)}=(\xi\otimes w,\nabla z)_{L^2_{\mu}(T^*M\otimes E)}
=(\nabla^{*,\mu}(\xi\otimes w),z)_{\mu}\nonumber\\
&=-(\operatorname{div}_{\mu}(X)w,z)_{\mu}-(\nabla_{X}w,z)_{\mu},
\end{align}
where, in local coordinates $x^{1},\,x^{2},\dots,x^{n}$, for $X=X^j\frac{\partial}{\partial x^{j}}$, with Einstein summation convention,
\[
\operatorname{div}_{\mu}(X):=\frac{1}{\kappa}\left(\frac{\partial}{\partial x^{j}}\left(\kappa X^{j}\right)\right).
\]
(Recall that $d\mu=\kappa(x)\,dx^{1}dx^{2}\dots dx^{n}$, where $\kappa(x)$ is a positive $C^{\infty}$-density.)
Since $X^{j}=(g^{TM})^{jl}\left(\gamma\frac{\partial\gamma}{\partial x^{l}}\right)$, we have
\begin{align}\label{E:t-4}
\operatorname{div}_{\mu}(X)=|d\gamma|^2-\gamma(\Delta_{M,\mu}\gamma),
\end{align}
where $|d\gamma(x)|$ is the norm of $d\gamma(x)\in T_{x}^*M$ induced by $g^{TM}$,
and $\Delta_{M,\mu}$ is as in~(\ref{E:cordes-lap}) with metric $g^{TM}$.
Combining~(\ref{E:t-1})--(\ref{E:t-4})  and noting that
\[
(d\gamma\otimes w,d\gamma\otimes z)_{L^2_{\mu}(T^*M\otimes E)}=\int_{M}|d\gamma|^2\langle w, z\rangle\,d\mu,
\]
we obtain
\begin{align}
&(H_1w,z)_{\mu_1}=\int_{M}\langle \nabla w,\nabla z\rangle \gamma^2\,d\mu+\int_{M}\langle V w, z\rangle \gamma^2\,d\mu+\int_{M}\gamma(\Delta_{M,\mu}\gamma)\langle w,z\rangle\,d\mu\nonumber\\
&=(\nabla w,\nabla z)_{L^2_{\mu_1}(T^*M\otimes E)}+(Vw,z)_{\mu_1}+(\gamma^{-1}(\Delta_{M,\mu}\gamma)w,z)_{\mu_1}\nonumber\\
&=(\nabla^{*,\mu_1}\nabla w, z)_{\mu_1}+(Vw,z)_{\mu_1}+(\gamma^{-1}(\Delta_{M,\mu}\gamma)w,z)_{\mu_1},
\end{align}
which shows~(\ref{E:op-H-1}).

By~(\ref{E:assumption-V-below-q}) and~(\ref{E:gamma-solution}) it follows that
\[
\widetilde{V}(x)=\frac{\Delta_{M,\mu}\gamma}{\gamma}\,\textrm{Id}(x)+V(x)\geq \,(C-1)\textrm{Id}(x),\qquad\textrm{for all }x\in M,
\]
where $C$ is as in~(\ref{E:lap-semi-bounded}).
Thus, by Theorem~\ref{T:main-2} the operator $(H_1)^{k}|_{\ecomp}$ is essentially self-adjoint in $L^2_{\mu_1}(E)$ for all $k\in\mathbb{Z}_{+}$.
$\hfill\square$

\section{Proof of Theorem~\ref{T:main-1}}\label{S:proof-main-1}
Throughout the section, we assume that the hypotheses of Theorem~\ref{T:main-1} are satisfied. In subsequent discussion, the notation $\widehat{D}$ is as in~(\ref{E:defsym}) and the operators $\hmin$ and $\hmax$ are as in Section~\ref{SS:mm-1}. We begin with the following lemma, whose proof is a direct consequence of the definition of $\hmax$ and local elliptic regularity.
\begin{lem}\label{L:domain-w-2-2} Under the assumption $V\in L^{\infty}_{\loc}(\End E)$,  we have the following inclusion:\\ $\Dom(\hmax)\subset W^{2,2}_{\loc}(E)$.
\end{lem}

The proof of the next lemma is given in Lemma 8.10 of \cite{B-M-S}.
\begin{lem}\label{L:shub} For any $u\in \Dom(H_{\max})$ and any Lipschitz function
with compact support $\psi\colon M\to\mathbb{R}$, we have:
\begin{equation}\label{E:shub}
   (D(\psi u),D(\psi u)) \ + \ (V\psi u, \psi u) \ = \ \RE(\psi Hu,\psi u)
          \ + \ \|\widehat{D}(d\psi)u\|^2.
\end{equation}
\end{lem}

\begin{cor}\label{C:nen}
Let $H$ be as in~(\ref{E:HV}), let $u\in L^2(E)$ be a weak solution of $Hu=0$, and let $\psi\colon M\to\mathbb{R}$ be a Lipschitz function with compact support.
Then
\begin{equation}\label{E:ute}
(\psi u, \,H (\psi u)) = \|\widehat{D}(d\psi)u\|^2,
\end{equation}
where $(\cdot,\cdot)$ on the left-hand side denotes the duality between $W_{\loc}^{1,2}(E)$ and $W^{-1,2}_{\comp}(E)$.
\end{cor}
\begin{proof} Since $u\in L^2(E)$ and $Hu=0$, we have $u\in\Dom(\hmax)\subset W^{2,2}_{\loc}(E)\subset W^{1,2}_{\loc}(E)$, where the first inclusion follows by Lemma~\ref{L:domain-w-2-2}. Since $\psi$ is a Lipschitz compactly supported function, we get $\psi u\in W^{1,2}_{\comp}(E)$ and, hence, $H(\psi u)\in W^{-1,2}_{\comp}(E)$.  Now the equality~(\ref{E:ute}) follows from~(\ref{E:shub}), the assumption $Hu=0$,  and
\[
(\psi u, \,H (\psi u)) \ = \ (\psi u, \, D^*D (\psi u)) \ + \ (V \psi u,\psi u) \ = \  (D(\psi u),D(\psi u)) \ + \ (V\psi u, \psi u),
\]
where in the second equality we used integration by parts; see Lemma 8.8 in~\cite{B-M-S}. Here, the two leftmost symbols $(\cdot,\cdot)$ denote the duality between $W_{\comp}^{1,2}(E)$ and $W^{-1,2}_{\loc}(E)$, while the remaining ones stand for $L^2$-inner products.
\end{proof}

The key ingredient in the proof of Theorem~\ref{T:main-1} is the Agmon-type estimate given in the next lemma, whose proof, inspired by an idea of~\cite{Nen}, is based on the technique developed in~\cite{Col-Tr} for magnetic Laplacians on an open
set with compact boundary in $\mathbb{R}^{n}$.

\begin{lem}\label{L:Hor}
Let $\lambda\in\mathbb{R}$ and let $v\in L^2(E)$ be a weak solution of $(H-\lambda)v=0$.
Assume that that there exists a constant $c_1>0$ such that, for all $u \in W_{\comp}^{1,2}(E)$,
\begin{equation}\label{E:bou}
(u,\, (H-\lambda) u )  \geq \lambda_0^2\int_{M}\max \left(\dfrac{1}{r(x)^2},1
\right)|u(x)|^2\,d\mu(x) +  c_1 \|u\|^2,
\end{equation}
where $r(x)$ is as in~(\ref{E:dist-boundary}), $\lambda_0$ is as in~(\ref{E:equalst}), the symbol $(\cdot,\cdot)$ on the left-hand side denotes the duality between $W_{\comp}^{1,2}(E)$ and $W^{-1,2}_{\loc}(E)$, and $|\cdot|$ is the norm in the fiber $E_{x}$.

Then, the following equality holds: $v=0$.
\end{lem}

\begin{proof}
Let $\rho$ and $R$ be numbers satisfying
$0< \rho < 1/2$ and  $ 1 < R < +\infty$.
For any $\varepsilon >0$, we define the  function $f_{\varepsilon}\colon M \rightarrow \mathbb{R}$
by $f_{\varepsilon}(x)=F_{\varepsilon}(r(x))$, where $r(x)$ is as in~(\ref{E:dist-boundary}) and $F_{\varepsilon}\colon[0,\infty)\to \mathbb{R}$ is the
continuous piecewise affine
function defined by
\[
F_{\varepsilon}(s)= \left\{
\begin{array}{l}
0  {\rm ~ for~}  s\leq \varepsilon   \\
\rho (s-\varepsilon)/(\rho  - \varepsilon  )   {\rm ~ for~} \varepsilon \leq s \leq  \rho  \\
s    {\rm ~ for~} \rho  \leq s \leq  1  \\
1   {\rm ~ for~} 1 \leq s \leq R  \\
 R+1 -s   {\rm ~ for~} R \leq s \leq R+1 \\
0   {\rm ~ for~ }  s \geq R+1.
\end{array}
\right.
\]

Let us fix $x_0\in M$. For any $\alpha >0$, we define the function $p_{\alpha}\colon M \to\mathbb{R}$
by
\[
p_{\alpha}(x)=P_{\alpha}(d_{g^{TM}}(x_0,x)),
\]
where  $P_{\alpha}\colon [0,\infty) \to \mathbb{R}$ is the
continuous piecewise affine function defined by
\[ P_{\alpha}(s)= \left\{
\begin{array}{l}
1  {\rm ~ for~}  s\leq 1/{\alpha}   \\
-{\alpha}  s + 2  {\rm ~ for~} 1/{\alpha}  \leq s \leq  2/{\alpha}  \\
0   {\rm ~ for~ }  s \geq 2/{\alpha}.
\end{array}
\right.
\]
Since $\widehat{d}_{g^{TM}}(x_0,x)\leq d_{g^{TM}}(x_0,x)$, it follows that the support of $f_{\varepsilon}p_{\alpha}$ is contained in the set $B_{\alpha}:=\{x\in M\colon \widehat{d}_{g^{TM}}(x_0,x)\leq 2/\alpha\}$. By assumption (A1) we know that $\widehat{M}$ is a geodesically complete Riemannian manifold. Hence, by Hopf--Rinow Theorem the set $B_{\alpha}$ is compact. Therefore, the support of $f_{\varepsilon}p_{\alpha}$ is compact. Additionally, note that
$f_{\varepsilon}p_{\alpha}$ is a $\beta$-Lipschitz function (with respect to the distance corresponding to the metric $g^{TM}$) with $\beta=\frac{{\rho}}{{\rho-\varepsilon}}+\alpha$.

Since $v\in L^2(E)$ and $(H-\lambda)v=0$, we have $v\in\Dom(\hmax)\subset W^{2,2}_{\loc}(E)\subset W^{1,2}_{\loc}(E)$, where the first inclusion follows by Lemma~\ref{L:domain-w-2-2}. Since $f_{\varepsilon}p_{\alpha}$ is a Lipschitz compactly supported function, we get $f_{\varepsilon}p_{\alpha} v\in W^{1,2}_{\comp}(E)$ and, hence, $((H-\lambda)(f_{\varepsilon}p_{\alpha} v))\in W^{-1,2}_{\comp}(E)$.

Using~(\ref{E:equalst}) we have
\begin{equation}\label{E:est-remaining}
\|\widehat{D}(d(f_{\varepsilon}p_{\alpha}))v\|^2\leq \lambda_0^2\int_{M}|d(f_{\varepsilon}p_{\alpha})(x)|^2|v(x)|^2\,d\mu(x),
\end{equation}
where $|d(f_{\varepsilon}p_{\alpha})(x)|$ is the norm of $d(f_{\varepsilon}p_{\alpha})(x)\in T_{x}^*M$ induced by $g^{TM}$.

By Corollary~\ref{C:nen} with $H-\lambda$ in place of $H$ and the inequality~(\ref{E:est-remaining}), we get
\begin{equation}\label{E:equ-rhs}
(f_{\varepsilon}p_{\alpha} v,\, (H-\lambda) (f_{\varepsilon}p_{\alpha} v)) \leq \lambda_0^2\left(\frac{{\rho}}{{\rho-\varepsilon}}+\alpha\right)^2\|v\|^2.
\end{equation}

On the other hand, using the definitions of $f_{\varepsilon}$ and $p_{\alpha}$ and the assumption~(\ref{E:bou}) we have
\begin{equation} \label{E:equ-lhs}
(f_{\varepsilon}p_{\alpha} v,\,(H-\lambda) (f_{\varepsilon}p_{\alpha}  v))
\geq \lambda_0^2\int_{S_{\rho,R,\alpha}}|v(x)|^2\,d\mu(x)+c_1 \| f_{\varepsilon}p_{\alpha} v
\|^2,
\end{equation}
where
\[
S_{\rho,R,\alpha}:=\{x\in M\colon \rho\leq r(x) \leq R \textrm{ and } d_{g^{TM}}(x_0,x) \leq 1/\alpha\}.
\]
In~(\ref{E:equ-lhs}) and (\ref{E:equ-rhs}), the symbol $(\cdot,\cdot)$ stands for the duality between $W_{\comp}^{1,2}(E)$ and $W^{-1,2}_{\loc}(E)$. We now combine (\ref{E:equ-lhs}) and (\ref{E:equ-rhs}) to get
\[
\lambda_0^2\int_{S_{\rho,R,\alpha}}|v(x)|^2\,d\mu(x) \ + \ c_1 \| f_{\varepsilon}p_{\alpha} v\|^2\leq \lambda_0^2\left(\frac{{\rho}}{{\rho-\varepsilon}}+\alpha\right)^2\|v\|^2.
\]
We fix $\rho$, $R$, and $\varepsilon$, and let $\alpha\to 0+$. After that we let $\varepsilon \to 0+$. The last step is to do $\rho \to 0+ $ and $R \to +\infty$.
As a result, we get $v=0$.
\end{proof}

\noindent\textbf{End of the proof of Theorem~\ref{T:main-1}.} Using integration by parts (see Lemma 8.8 in~\cite{B-M-S}),  we have
\[
(u,\, Hu) \ = \ (u,D^*Du) \ + \ (Vu,u) \ = \ (Du,Du) \ + \ (Vu,u) \ \geq \ (Vu,u), \qquad\textrm{ for all }u\in W_{\comp}^{1,2}(E),
\]
where the two leftmost symbols $(\cdot,\cdot)$ denote the duality between $W_{\comp}^{1,2}(E)$ and $W^{-1,2}_{\loc}(E)$, while the remaining ones stand for $L^2$-inner products.
Hence, by assumption~(\ref{E:potential-minorant}) we get:
\begin{align}\label{E:bou-new}
(u,\, (H-\lambda) u) &\geq \lambda_0^2\int_{M}\dfrac{1}{r(x)^2}|u(x)|^2\,d\mu(x)-(\lambda+C)\|u\|^2\nonumber\\
&\geq  \lambda_0^2\int_{M}\max \left(\dfrac{1}{r(x)^2},1
\right)|u(x)|^2\,d\mu(x) -(\lambda+C+1)\|u\|^2.
\end{align}
Choosing, for instance, $\lambda=-C-2$ in~(\ref{E:bou-new}) we get the
inequality~(\ref{E:bou}) with $c_1=1$.

Thus, $\hmin-\lambda$ with $\lambda=-C-2$ is a symmetric operator satisfying
$(u,\, (\hmin-\lambda) u)\geq \|u\|^2$, for all $u\in\ecomp$. In this case, it is known (see Theorem X.26 in~\cite{rs}) that the essential self-adjointness of $\hmin-\lambda$ is equivalent to the following statement: if $v\in L^2(E)$ satisfies $(H-\lambda)v=0$, then $v=0$. Thus, by Lemma~\ref{L:Hor}, the operator $(\hmin-\lambda)$ is essentially self-adjoint. Hence, $\hmin$ is essentially self-adjoint. $\hfill\square$

\end{document}